\documentclass[]{interact}

\usepackage{epstopdf}
\usepackage[caption=false]{subfig}

\usepackage[numbers,sort&compress]{natbib}
\bibpunct[, ]{[}{]}{,}{n}{,}{,}
\renewcommand\bibfont{\fontsize{10}{12}\selectfont}

\theoremstyle{plain}
\newtheorem{theorem}{Theorem}[section]
\newtheorem{lemma}[theorem]{Lemma}

\newtheorem{proposition}[theorem]{Proposition}

\theoremstyle{definition}

\theoremstyle{remark}
\newtheorem{remark}{Remark}

\newtheorem{assumption}{Assumption}

\DeclareMathOperator{\proj}{proj}
\DeclareMathOperator{\grad}{grad}
\DeclareMathOperator{\R}{R}
\DeclareMathOperator{\T}{T}

\usepackage{comment}
\usepackage[ruled,linesnumbered,vlined,algo2e]{algorithm2e}
\usepackage{url}
\usepackage{multirow}
\usepackage{color}

\begin{document}

\articletype{ARTICLE}

\title{Riemannian Gradient Method with Momentum}

\author{
\name{Filippo Leggio\textsuperscript{a} and Diego Scuppa\textsuperscript{a}\thanks{CONTACT Diego Scuppa. Email: diego.scuppa@uniroma1.it}}
\affil{\textsuperscript{a}Department of Computer, Control and Management Engineering, Sapienza University of Rome, Rome, Italy}
}

\maketitle

\begin{abstract}
In this paper, we consider the problem of minimizing a smooth function on a Riemannian manifold and present a Riemannian gradient method with momentum. The proposed algorithm represents a substantial and nontrivial extension of a recently introduced method for unconstrained optimization. We prove that the algorithm, supported by a safeguarding rule, produces an $\epsilon$-stationary point with a worst-case complexity bound of $\mathcal{O}(\epsilon^{-2})$.
Extensive computational experiments on benchmark problems are carried out, comparing the proposed method with state-of-the-art solvers available in the Manopt package. The results demonstrate competitive and often superior performance.
Overall, the numerical evidence confirms the effectiveness and robustness of the proposed approach, which provides a meaningful extension of the recently introduced momentum-based method to Riemannian optimization.

\end{abstract}


\begin{amscode}
	90C26; 90C30; 90C53; 90C60
\end{amscode}

\section{Introduction}

In this work, we focus on the problem
\begin{equation*}
\min\{f(x) : x \in \mathcal{M}\},
\end{equation*}
where $\mathcal M$ is a \textit{Riemannian submanifold} of a finite-dimensional Euclidean space $\mathcal E$ \cite{RiemannOptimization,SpringerSato,
MatrixManifolds,gabay1982}, and $f : \mathcal{M}\to\mathbb{R}$ is a smooth and non-convex function.

Riemannian manifolds frequently arise in practical applications, such as machine learning \cite{lin2008,tuzel2008}, radar communication \cite{liu2018}, low-rank matrix (and tensor) completion \cite{vandereycken2013,kressner2013,boumal-absil2015}, invariant subspace computation \cite{Absil2004}, semidefinite programming \cite{li2023}, analysis of shape spaces \cite{ring-wirth2012}, and synchronization over the special Euclidean group \cite{rosen2016}.

Several well-known methods for standard unconstrained optimization in the Euclidean space $\mathcal E=\mathbb{R}^n$ \cite{SciandroneGrippo, bertsekas2016} have been adapted to the Riemannian setting, such as conjugate gradient methods \cite{sato2022,boumal-absil2015,vandereycken2013}, trust-region methods \cite{absil2007,boumal-absil2015,boumal2015}, Barzilai--Borwein methods \cite{iannazzo-porcelli2017}, and L-BFGS methods \cite{huang2018}. 
Convergence results for these methods have been extensively analyzed in the literature, with several studies providing global convergence rates. For instance, under Lipschitz-type assumptions on $f$, \cite{complexity_cartis} provided worst-case complexity bounds for the Riemannian gradient descent and the Riemannian trust-region methods. Very recently, in \cite{zhou2026} it was proved that the Riemannian conjugate gradient method produces an $\epsilon$-stationary point with an iteration complexity of $\mathcal O(\epsilon^{-2})$ under a non-monotone line-search, together with a restart strategy that replaces the computed direction with the negative Riemannian gradient, rescaled by the Barzilai--Borwein factor \cite{barzilai-borwein1988,raydan1997}, whenever certain conditions are not satisfied. These conditions can be interpreted as a generalization of the \textit{gradient-related} conditions proposed for unconstrained optimization in \cite{cartis2015}.  

In this work, we investigate Riemannian gradient methods with momentum, namely iterative schemes in which the update rule combines the Riemannian gradient with a suitable contribution from the previous search direction.

Riemannian gradient methods with momentum have been recently studied in \cite{RegolarizzazioneCubica}. The algorithm proposed therein computes, at each iteration, the solution of a cubic-regularized subproblem restricted to a suitably chosen subspace of the tangent space, and is coupled with a (possibly non-monotone) line-search procedure. Under mild assumptions involving second-order information, the authors establish a worst-case iteration complexity bound of $\mathcal{O}(\epsilon^{-2})$.

In the present work, drawing inspiration from the recent papers \cite{lapucci2025} and \cite{lapucci2026}, we propose a line-search-based method in which, at each iteration, the search direction is obtained by solving a quadratic bidimensional subproblem involving a linear combination of two directions in the tangent space.
The core of the proposed method lies in the construction of the quadratic subproblem, which must account for technical and computational issues related to the fact that the feasible set is a Riemannian manifold.
Under standard assumptions, we prove that the algorithm, supported by a suitable restart strategy, produces an $\epsilon$-stationary point with an iteration complexity of $\mathcal O(\epsilon^{-2})$.

The main contribution of this paper is the development of a novel first-order algorithm for optimization on manifolds, equipped with rigorous global convergence guarantees and competitive computational complexity. Its practical effectiveness is corroborated by extensive numerical experiments demonstrating performance on par with, or superior to, current state-of-the-art methods.

The remainder of the paper is organized as follows. Section \ref{sec:preliminari} provides a summary of fundamental concepts of Riemannian manifolds and Riemannian optimization. Section \ref{sec:convergenza} presents convergence results for a general method based on \textit{gradient-related} directions. Section \ref{sec:Rn} summarizes the approach described in \cite{lapucci2025}, whose Riemannian adaptation is discussed in Section \ref{sec:generalizzazione}. In Section \ref{sec:approccio} we provide some algorithmic considerations and present the proposed convergent algorithm. Section \ref{sec:esperimenti} discusses numerical results, and Section \ref{sec:conclusioni} concludes the paper with final remarks.

\section{Preliminary technical results on Riemannian manifolds}
\label{sec:preliminari}

In this section, we briefly review some basic concepts from Riemannian optimization; see, e.g.,~\cite{RiemannOptimization}.
Let $\mathcal{E}$ be a finite-dimensional Euclidean space of dimension $n$, endowed with an inner product $\langle \cdot, \cdot \rangle$. Typical examples include $\mathcal{E} = \mathbb{R}^n$, with $\langle x, y \rangle = x^\top y$, or $\mathcal{E} = \mathbb{R}^{n_1 \times n_2}$, with $\langle X, Y \rangle = \mathrm{tr}(X^\top Y)$.  
A \emph{smooth embedded manifold} $\mathcal{M}$ is a subset of $\mathcal{E}$ such that, for each point $x \in \mathcal{M}$, there exists a local linear approximation of $\mathcal{M}$ at $x$, called the \emph{tangent space} and denoted by $\T_x \mathcal{M}$.
If $\mathcal{M}$ is defined as
\begin{equation}
\label{eqn:manLICQ}
\mathcal{M} = \{ x \in \mathcal{E} : h_1(x) = 0, \dots, h_m(x) = 0 \},
\end{equation}
where $m < n$ and $h_1, \ldots, h_m : \mathcal{E} \to \mathbb{R}$ are smooth functions satisfying the Linear Independence Constraint Qualification (LICQ), then the tangent space at $x$ is given by
$\T_x \mathcal{M}
= \left\{ v \in \mathcal{E} : \langle \nabla h_i(x), v \rangle = 0, \; i = 1, \ldots, m \right\}$.

Each tangent space $\T_x \mathcal{M}$ can be endowed with the inner product inherited from $\mathcal{E}$. With this choice of metric, $\mathcal{M}$ is called a \emph{Riemannian submanifold} of $\mathcal{E}$.

Let $f : \mathcal{M} \to \mathbb{R}$ be a continuously differentiable function. The \emph{Riemannian gradient} of $f$ at $x \in \mathcal{M}$ is defined as the unique vector $\grad f(x) \in \T_x \mathcal{M}$ such that
$\langle \grad f(x), v \rangle = \mathrm{D} f(x)[v]$ for all $v \in \T_x \mathcal{M}$.
It can be shown that
$\grad f(x) = \proj_x(\nabla f(x))$,
where $\nabla f(x)$ denotes the Euclidean gradient of $f$ and $\proj_x : \mathcal{E} \to \T_x \mathcal{M}$ is the orthogonal projection onto the tangent space.

The \emph{Riemannian Hessian} of $f$ at $x$ is defined as the self-adjoint linear operator
$\mathrm{Hess}\, f(x) : \T_x \mathcal{M} \to \T_x \mathcal{M}$,
given by
$\mathrm{Hess}\, f(x)[v] = \proj_x \big( \mathrm{D} \grad f(x)[v] \big)$.

For each $x \in \mathcal{M}$, a \emph{retraction} is a smooth mapping $\R_x : \T_x \mathcal{M} \to \mathcal{M}$ satisfying
$\R_x(0) = x$ and 
$\mathrm{D} \R_x(0) = \mathrm{Id}$.

Moreover, for any $x, y \in \mathcal{M}$, a \emph{vector transport} is a linear map
$\T_{y \leftarrow x} : \T_x \mathcal{M} \to \T_y \mathcal{M}$
such that $\T_{x \leftarrow x}$ is the identity on $\T_x \mathcal{M}$. For our purposes, a practical choice may be given by the orthogonal projection onto $\T_y \mathcal{M}$, restricted to $\T_x \mathcal{M}$, namely
$\T_{y \leftarrow x} := \proj_{y|_{\T_x \mathcal{M}}}$.

Riemannian optimization algorithms aim to minimize a smooth function $f$ over $\mathcal{M}$.
Using a first-order Taylor expansion of $f$, a necessary condition for a point $x^* \in \mathcal{M}$ to be a local minimizer is that it is a stationary point, namely
$\grad f(x^*) = 0$.
Furthermore, if $\R_x$ is a second-order retraction and $f$ is twice continuously differentiable, a second-order Taylor expansion yields the condition
$\langle v, \mathrm{Hess}\, f(x^*)[v] \rangle \ge 0$
 for all $v \in \T_{x^*} \mathcal{M}$.
When $\mathcal{M}$ is defined as in~\eqref{eqn:manLICQ}, these conditions are equivalent to the first- and second-order Karush-Kuhn-Tucker optimality conditions.

\section{A convergent first-order framework}
\label{sec:convergenza}

In this section, we present a general line-search-based framework with guaranteed convergence properties. The theoretical results rely on the use of a sequence of \emph{gradient-related} search directions and an Armijo line-search strategy.

We note that the results of this section follow from the more general analysis in~\cite{zhou2026}. For completeness, we provide here the statement, while the proof is reported in Appendix \ref{app:proofs}.

Let $\mathcal{M}$ be a smooth Riemannian submanifold of a finite-dimensional Euclidean space $\mathcal{E}$, and let $f : \mathcal{M} \to \mathbb{R}$ be a continuously differentiable function. We do not assume $f$ to be convex; however, we assume that $f$ is bounded from below:

\begin{assumption}
\label{ass:flow}
There exists $f_{\mathrm{low}} \in \mathbb{R}$ such that
$f(x) \ge f_{\mathrm{low}}$ for all $x \in \mathcal{M}$.
\end{assumption}

Consider the following optimization problem:
\begin{equation*}
   \min\limits_{x \in \mathcal{M}} f(x).
\end{equation*}
A general first-order Riemannian algorithm generates a sequence of points $\{x_k\}\subset \mathcal M$ possibly converging to a stationary point $x^*\in \mathcal M$, that is, such that 
\begin{equation*}
    \grad f(x^*)=0.
\end{equation*}
The general updating rule can be summarized as follows: given a point $x_k\in \mathcal M$, find a direction $d_k \in \T_{x_k}\mathcal M$ and a step size $\eta_k>0$ (via a line-search), and compute a new point in $\mathcal M$ as
\begin{equation*}
x_{k+1} = \R_{x_k}(\eta_k d_k),
\end{equation*}
where $\R_{x_k}$ is a retraction.
We say that a sequence of directions $\{d_k\}$ is \emph{gradient-related} to the sequence $\{x_k\}$ if the following assumption holds.

\begin{assumption}
\label{ass:gradientrelated}
There exist constants $0<c_1 \le c_2$ such that, for all $k$,
\begin{subequations} \label{eqn:gradrelated}
\begin{equation}
\langle \grad f(x_k), d_k \rangle \le - c_1 \|\grad f(x_k)\|^2,\label{eqn:gradrelated1}
\end{equation}
\begin{equation}
\|d_k\| \le c_2 \|\grad f(x_k)\|.\label{eqn:gradrelated2}
\end{equation}
\end{subequations}
\end{assumption}

Note that the notion of gradient-related directions extends the corresponding concept introduced for unconstrained optimization \cite{cartis2015}. Moreover, we assume the following.

\begin{assumption}
\label{ass:assunzione}
There exists $L>0$ and, for all $k$, there exists $\rho_k>\|d_k\|$ (possibly $\rho_k=\infty$ for all $k$) such that:
\begin{enumerate}
\item the retraction $\R_{x_k}$ is defined at least on the ball $B(\rho_k)\cap\T_{x_k}\mathcal{M}$;
\item for all $d\in B(\rho_k)\cap \T_{x_k}\mathcal{M}$,
\begin{equation}
\label{eqn:lipschitz}
| f(\R_{x_k}(d)) - f(x_k) - \langle \grad f(x_k), d \rangle | \le \frac{L}{2}\|d\|^2.
\end{equation}
\end{enumerate}
\end{assumption}

We remark that, in the Riemannian optimization setting, one could introduce a notion of Lipschitz continuity for the Riemannian gradient. However, as also discussed in~\cite{complexity_cartis}, it is more convenient to directly assume the Lipschitz-type condition~\eqref{eqn:lipschitz}.

We now discuss the theoretical properties of a general first-order algorithm in which the sequence of search directions $\{d_k\}$ is gradient-related to the iterates $\{x_k\}$, as described in Algorithm~\ref{alg:gradrel}. Moreover, we assume that an Armijo line-search procedure with initial step size equal to $1$ is employed (see Algorithm~\ref{alg:armijo}).

\begin{algorithm2e}[ht]
  \SetAlCapHSkip{0.22cm}
  \caption{First-order algorithm with gradient-related directions}
  \label{alg:gradrel}
  \SetInd{0.1cm}{0.5cm}
  \SetVlineSkip{0.1cm}
  \SetNlSkip{0.2cm}
  \SetNoFillComment
  \KwIn{$x_0\in\mathcal{M}$, $\gamma\in(0,1)$, $\delta\in(0,1)$, $\epsilon>0$, $0<c_1\le c_2$}
  \KwOut{approximate stationary point $x_k$}
  $k \leftarrow 0$\;
  \While{$\|\grad f(x_k)\| > \epsilon$}{
    choose $d_k \in \T_{x_k}\mathcal M$ such that 
    $\langle \grad f(x_k), d_k\rangle\le -c_1\|\grad f(x_k)\|^2$ and
    $\|d_k\|\le c_2\|\grad f(x_k)\|$\; \label{line:dk}
    $\eta_k \leftarrow \texttt{Armijo}(x_k, f(x_k), \grad f(x_k), d_k, 1, \gamma, \delta)$\;
    $x_{k+1} \leftarrow R_{x_k}(\eta_k d_k)$\;
    $k \leftarrow k + 1$\;
  }
\end{algorithm2e}

\begin{algorithm2e}[ht]
  \SetAlCapHSkip{0.22cm}
  \caption{Armijo line-search}
  \label{alg:armijo}
  \SetInd{0.1cm}{0.5cm}
  \SetVlineSkip{0.1cm}
  \SetNlSkip{0.2cm}
  \SetNoFillComment
  \SetKwFunction{FArmijo}{Armijo}
  \SetKwProg{Fn}{Function}{:}{}
  \Fn{\FArmijo{$x_k, f(x_k), \grad f(x_k), d_k, \eta_k, \gamma, \delta$}}{
    \While{$f\big(R_{x_k}(\eta_k d_k)\big) > f(x_k) + \gamma\,\eta_k\,\langle \grad f(x_k), d_k\rangle$}{
      $\eta_k \leftarrow \delta\,\eta_k$\;
    }
    \Return $\eta_k$\;
  }
\end{algorithm2e}

\begin{proposition}
\label{prop:Convergenza}
Suppose that Assumptions \ref{ass:flow}--\ref{ass:gradientrelated}--\ref{ass:assunzione} are satisfied (with constants $f_{\mathrm{low}}\in \mathbb R$, $0 < c_1\le c_2$, and $L>0$). Let $\delta\in(0,1)$ and $\gamma\in(0,1)$. Let $\{x_k\}$ and $\{d_k\}$ be the sequences produced by Algorithm~\ref{alg:gradrel}. Assume that $d_k$ is computed without additional function or gradient evaluations. For any $\epsilon>0$, define:
\begin{itemize}
\item $k_{\epsilon}$: the first iteration such that $x_{k_\epsilon}$ satisfies $\|\grad f(x_{k_\epsilon})\|\le \epsilon$;
\item $nf_{\epsilon}$: the number of function evaluations (and retractions) required to compute $x_{k_\epsilon}$;
\item $ni_{\epsilon}$: the total number of iterations $k$ such that $\|\grad f(x_k)\|> \epsilon$.
\end{itemize}
Then, in the worst case, Algorithm~\ref{alg:gradrel} guarantees:
\begin{enumerate}
\item if $L\le2(1-\gamma)\frac{c_1}{c_2^2}$, then
\begin{equation}
\label{eqn:caso1}
k_\epsilon=nf_{\epsilon}\le \frac {f(x_0) - f_{\mathrm{low}}}{ \gamma c_1}\, \epsilon^{-2}={\cal O}(\epsilon^{-2});
\end{equation}
\item otherwise,
\begin{subequations}
\begin{equation}
\label{eqn:caso2a}
k_\epsilon\le \frac {L c_2^2\big(f(x_0) - f_{\mathrm{low}}\big)}{2 \delta \gamma(1-\gamma) c_1^2}\, \epsilon^{-2}={\cal O}(\epsilon^{-2}),
\end{equation}
\begin{equation}
\label{eqn:caso2b}
nf_{\epsilon}\le \frac {L c_2^2\big(f(x_0) - f_{\mathrm{low}}\big)}{2 \delta \gamma(1-\gamma) c_1^2}
\bigg(\big\lfloor \log_{\delta} \Big(\frac{2}{L} (1-\gamma)\frac{c_1}{c_2^2}\Big)\big\rfloor +2\bigg)\,
\epsilon^{-2}
= {\cal O}(\epsilon^{-2}).
\end{equation}
\end{subequations}
\end{enumerate}
Moreover, $ni_{\epsilon}$ satisfies the same bounds as in \eqref{eqn:caso1}--\eqref{eqn:caso2a}.
\end{proposition}

A proof of Proposition~\ref{prop:Convergenza} is reported in Appendix~\ref{app:proofs}.  

We observe that the assumptions adopted in this work are milder than those in \cite[Assumption~2.1]{RegolarizzazioneCubica}, where the authors established, for a practical algorithm, a worst-case complexity of $\mathcal O(\epsilon^{-2})$ under assumptions involving second-order information.

\section{An unconstrained gradient method with momentum}
\label{sec:Rn}

In this section, we briefly summarize the gradient method with momentum developed in~\cite{lapucci2025} for the Euclidean setting $\mathbb{R}^n$. Let $f : \mathbb{R}^n \to \mathbb{R}$ be a continuously differentiable function, not necessarily convex, and consider the unconstrained optimization problem
\[
\min_{x \in \mathbb{R}^n} f(x).
\]
The method generates a sequence of iterates $\{x_k\}$ that is globally convergent to stationary points of $f$.
Given the current iterate $x_k \in \mathbb{R}^n$, the next iterate is computed as
\[
x_{k+1} = x_k + \eta_k d_k,
\]
where $\eta_k > 0$ is determined by an Armijo line-search procedure, and the search direction $d_k$ is defined as
\[
d_k = -\alpha_k g_k + \beta_k s_k,
\]
with $g_k = \nabla f(x_k)$ and $s_k = x_k - x_{k-1}$. The coefficients $\alpha_k$ and $\beta_k$ are obtained by minimizing a quadratic model that approximates the function $f$ around $x_k$, namely
\begin{align*}
& \min_{d,\alpha,\beta} \; f(x_k) + g_k^\top d + \frac{1}{2} d^\top B_k d \\
& \text{s.t. } d = -\alpha g_k + \beta s_k,
\end{align*}
where $B_k \in \mathbb{R}^{n \times n}$ is a symmetric and positive-definite matrix.
This problem is equivalent to minimizing the following two-dimensional quadratic model:
\[
\min_{u \in \mathbb{R}^2} \; T_k^\top u + \frac{1}{2} u^\top H_k u,
\]
where
\[
T_k =
\begin{bmatrix}
-\|g_k\|^2 \\
g_k^\top s_k
\end{bmatrix},
\qquad
u =
\begin{bmatrix}
\alpha \\
\beta
\end{bmatrix},
\]
and
\begin{equation}
\label{eqn:HkBk}
H_k =
\begin{bmatrix}
g_k^\top B_k g_k & -g_k^\top B_k s_k \\
- g_k^\top B_k s_k & s_k^\top B_k s_k
\end{bmatrix}.
\end{equation}

The authors of~\cite{lapucci2025} propose two alternative strategies:
\begin{itemize}
\item defining the $2 \times 2$ matrix $H_k$ from the $n \times n$ matrix $B_k$ as in~\eqref{eqn:HkBk};
\item directly selecting any symmetric $2 \times 2$ matrix $H_k$, independently of $B_k$.
\end{itemize}
Convergence conditions are established for both approaches by enforcing the search directions $\{d_k\}$ to be \emph{gradient-related}, that is,
\[
g_k^\top d_k \le -c_1 \|g_k\|^2,
\qquad
\|d_k\| \le c_2 \|g_k\|,
\quad \text{for all } k,
\]
for some constants $c_1, c_2 > 0$.
In particular, it is shown that if the level set
$\mathcal{L}_0 := \{ x \in \mathbb{R}^n : f(x) \le f(x_0) \}$
is compact, then the method either produces a stationary point $x_\nu$ of $f$ in a finite number of iterations or generates an infinite sequence $\{x_k\}$ admitting limit points, all of which are stationary points of $f$. Moreover, if $\nabla f$ is Lipschitz continuous on $\mathbb{R}^n$, then the method finds a point $x_k$ such that $\|\nabla f(x_k)\| \le \epsilon$ in at most $\mathcal{O}(\epsilon^{-2})$ iterations, function evaluations, and gradient evaluations.

A key practical issue concerns the choice of the matrix $B_k$, or equivalently of $H_k$. Several options are discussed in~\cite{lapucci2025}, including the computation of the three entries of $H_k$ by enforcing interpolation conditions of the quadratic model at three distinct points. This approach, however, requires additional evaluations of the objective function. Therefore, it is not suitable for extension to optimization problems on manifolds, as it would involve retraction operations that may be computationally expensive. Our proposed strategy to overcome this issue will be presented later.

\section{Issues behind the generalization to the Riemannian case}
\label{sec:generalizzazione}

In this section, we aim to adapt the gradient method with momentum to the Riemannian setting, discussing the technical issues arising in this generalization.

In this framework, given a point $x_k \in \mathcal M$, a new point is generated as $x_{k+1} = \R_{x_k}(\eta_k d_k)$, where $\eta_k > 0$ is determined by an Armijo line-search procedure and $d_k$ is a direction belonging to the tangent space $\T_{x_k}\mathcal M$. Specifically, $d_k$ is defined as a combination of two elements of $\T_{x_k}\mathcal M$: the Riemannian gradient
\[
g_k = \grad f(x_k) = \proj_{x_k}\nabla f(x_k),
\]
and a momentum term $s_k$. In order to ensure that $s_k$ belongs to $\T_{x_k} \mathcal M$, it cannot be defined as the difference of two consecutive iterates $x_k - x_{k-1}$, as in the Euclidean case. Instead, we define
\[
s_k = \T_{x_k \leftarrow x_{k-1}}(\eta_{k-1} d_{k-1}),
\]
that is, by transporting the previous search direction into the tangent space $\T_{x_k}\mathcal M$. As a practical choice for the vector transport, we consider the orthogonal projection onto $\T_{x_k}\mathcal M$, namely
\[
s_k = \proj_{x_k}(\eta_{k-1} d_{k-1}).
\]

The search direction is then obtained as
\begin{equation}
\label{eqn:dk}
d_k = -\alpha_k g_k + \beta_k s_k,
\end{equation}
for some parameters $\begin{bmatrix} \alpha_k \\ \beta_k \end{bmatrix} \in \mathbb R^2$.

To compute $\alpha_k$ and $\beta_k$, as in the Euclidean case \cite{lapucci2025} summarized in Section~\ref{sec:Rn}, we consider minimizing a quadratic approximation of the function $f$ around the point $x_k$:
\begin{align}
\label{eqn:MinProblem}
\begin{split}
    &\min_{d \in \T_{x_k}\mathcal{M}} f(x_k) + \langle g_k, d \rangle + \frac{1}{2} \langle d, B_k[d] \rangle \\
    &\text{s.t.} \quad d = -\alpha g_k + \beta s_k,
\end{split}
\end{align}
where $\langle \cdot,\cdot\rangle$ denotes the inner product on the tangent space $\T_{x_k}\mathcal{M}$ inherited from $\mathcal E$, and $B_k$ is a linear, self-adjoint, and positive-definite operator mapping $\T_{x_k}\mathcal{M}$ into itself. Note that, if $f$ is twice continuously differentiable, $B_k$ is chosen as $\text{Hess} f(x_k)$, and $\R_{x_k}$ is a second-order retraction, then the model in \eqref{eqn:MinProblem} corresponds to the second-order Taylor approximation of $f$ around $x_k$.

As in the Euclidean case, the problem can be reformulated as a minimization problem in $\mathbb R^2$:
\begin{equation}
\label{eqn:minimizationProblem}
    \min_{u \in \mathbb{R}^2} T_k^{\top} u + \frac{1}{2} u^{\top} H_k u,
\end{equation}
where
\[
T_k =
\begin{bmatrix}
-\|g_k\|^2 \\
\langle g_k, s_k \rangle
\end{bmatrix},
\qquad
u =
\begin{bmatrix}
\alpha \\
\beta
\end{bmatrix},
\]
and
\begin{equation}
\label{eqn:Hk}
H_k =
\begin{bmatrix}
\langle g_k, B_k[g_k]\rangle & -\langle g_k, B_k[s_k]\rangle \\
-\langle g_k, B_k[s_k]\rangle & \langle s_k, B_k[s_k]\rangle
\end{bmatrix}.
\end{equation}

Once a solution $u_k = \begin{bmatrix} \alpha_k \\ \beta_k \end{bmatrix}$ of~\eqref{eqn:minimizationProblem} is computed, the direction $d_k$ is defined as in~\eqref{eqn:dk}. A major issue concerns the choice of the operator $B_k$ (or, equivalently, of the $2 \times 2$ matrix $H_k$) in an efficient way, possibly without requiring additional retractions or function and gradient evaluations. For instance, approximating the Riemannian Hessian $\text{Hess} f(x)[v]$ via finite differences for some $v \in \T_x\mathcal M$ requires an additional retraction and gradient evaluation. Moreover, the direction $d_k$ must be enforced to be \emph{gradient-related} in order to guarantee convergence.

\begin{remark}
\label{oss:EsistenzaSoluzione}
The fact that the operator $B_k$ is self-adjoint and positive-definite on $\T_{x_k}\mathcal{M}$ ensures that, if $g_k$ and $s_k$ are nonzero, problem~\eqref{eqn:MinProblem} admits an optimal solution. A proof can be obtained as in~\cite[][Proposition~2]{lapucci2025}, with minor adaptations. Moreover, if $g_k$ and $s_k$ are linearly independent, then the $2 \times 2$ matrix $H_k$ defined in~\eqref{eqn:Hk} is full rank and positive-definite, and hence the solution is unique. Furthermore, if the operators $\{B_k\}$ have uniformly bounded eigenvalues, then the sequence of directions $\{d_k\}$ computed via~\eqref{eqn:dk} is gradient-related to $\{x_k\}$, and Assumption~\ref{ass:gradientrelated} is satisfied. A proof follows arguments similar to those in~\cite[][Proposition~3]{lapucci2025}.
\end{remark}

\section{The proposed approach}
\label{sec:approccio}

In this section, we discuss some practical choices for the Riemannian gradient method with momentum. Moreover, we establish the global convergence of the proposed approach, which directly follows from the discussion in Section \ref{sec:convergenza}, as we choose a particular set of directions $\{d_k\}$ that is gradient-related to the sequence of points $\{x_k\}$.

The scheme for the Riemannian gradient method with momentum is the same as Algorithm \ref{alg:gradrel}, with the direction $d_k$ in line \ref{line:dk} found according to these steps:
\begin{itemize}
\item choose a self-adjoint and positive-definite operator $B_k$, define the $2 \times 2$ symmetric and positive-definite matrix $H_k$ (as in \eqref{eqn:Hk}), and consider the quadratic minimization problem \eqref{eqn:minimizationProblem};
\item  find a solution $[\alpha_k, \beta_k]^{\top}$ of \eqref{eqn:minimizationProblem} and find the direction as $d_k=-\alpha_kg_k+\beta_k s_k$;
\item if $d_k$ is not gradient-related define a new gradient-related direction $d_k$ belonging to $\T_{x_k}\mathcal M$ according to some rules.
\end{itemize}

Two major issues are:
\begin{enumerate}
\item how to properly and efficiently define the operator \(B_k\), and, hence, the matrix $H_k$ in problem \eqref{eqn:minimizationProblem}?
\item how to choose a new gradient-related direction if the found one does not satisfy this condition?
\end{enumerate}

\subsection{Practical choice of operator $B_k$}
\label{sec:approxHk}

In this subsection, we provide an efficient way to choose matrix \(H_k\) in the problem (\ref{eqn:minimizationProblem}):
\[
H_k =
\begin{bmatrix}
\langle g_k, B_k[g_k]\rangle & -\langle g_k, B_k[s_k]\rangle \\
-\langle g_k, B_k[s_k]\rangle & \langle s_k, B_k[s_k]\rangle
\end{bmatrix}.
\]
We aim to consider a self-adjoint and positive-definite operator $B_k$ that should resemble the Riemannian Hessian $\text{Hess}f(x_k)$ in the case of $f$ twice continuously differentiable.
Note that it is not required to explicitly compute the operator $B_k$, but only its applications $B_k[g_k]$ and $B_k[s_k]$. We could use a finite differences approximation of the Hessian or the interpolation technique used in \cite{lapucci2025,lapucci2026}; however, to avoid extra gradient or function evaluations, we follow a different approach. Consider the element in the tangent space $\T_{x_k}\mathcal M$ defined as
\begin{equation*}
y_k=g_k-\T_{x_{k}\leftarrow x_{k-1}}(g_{k-1}).
\end{equation*}
Assume that the self-adjoint and positive-definite operator $B_k$ satisfies the secant equation 
\begin{equation*}
B_k[s_k]=y_k.
\end{equation*}
Then, defined $\rho_k=\langle g_k, B_k[g_k]\rangle$, we have to solve the system
\begin{equation}
\label{eqn:system2x2}
\begin{bmatrix}
\rho_k & -\langle g_k, y_k\rangle \\
-\langle g_k, y_k\rangle & \langle s_k, y_k\rangle
\end{bmatrix}
\begin{bmatrix}
\alpha \\ \beta    
\end{bmatrix} = \begin{bmatrix}\|g_k\|^2 \\ -\langle g_k, s_k \rangle \end{bmatrix}.
\end{equation}
We observe that the above system admits solution provided that the operator $B_k$ is self-adjoint, positive-definite, and $g_k$ and $s_k$ are linearly independent, Indeed, we have that the $2 \times 2$ matrix is positive-definite  (cf. Remark \ref{oss:EsistenzaSoluzione}).

For a practical choice of $B_k$, we follow the approach suggested by Yuan \& Stoer for the unconstrained optimization \cite{yuan-stoer1995}, where, assuming $s_k^\top y_k>0$, the scaled memoryless BFGS update for the $n\times n$ matrix $B_k$ is considered, i.e.,
\begin{equation}\label{eqn:Bk}
B_k =
\frac{1}{\lambda_k}
\left( I - \frac{s_k s_k^\top}{s_k^\top s_k} \right)
+ \frac{y_k y_k^\top}{s_k^\top y_k},
\end{equation}
being the scalar $\lambda_k > 0$  chosen according to the Barzilai--Borwein spectral gradient method~\cite{barzilai-borwein1988,raydan1997}, for example,
\[
\lambda_k = \frac{\|s_k\|^2}{s_k^\top y_k}.
\]
The  matrix $B_k$ defined by (\ref{eqn:Bk}) is symmetric and positive-definite, satisfies the secant equation $B_k s_k = y_k$, and does not require additional function or gradient evaluations to be constructed.

The extension to the case of Riemannian optimization is straightforward.
Assume that
\begin{equation}
\label{eqn:sy}
\langle s_k, y_k\rangle>0. 
\end{equation} 
Taking into account (\ref{eqn:Bk}),
define the operator $B_k:\T_{x_k}\mathcal M \to \T_{x_k}\mathcal M$ as follows: for every $d \in \T_{x_k}\mathcal M$,
\begin{equation}
\label{eqn:BFGS}
B_k[d] = \frac{1}{\lambda_k}\bigg(d-\frac{ \langle s_k, d\rangle}{\|s_k\|^2}s_k\bigg)+\frac{\langle y_k, d\rangle}{\langle s_k, y_k\rangle} y_k,
\end{equation}
with $\lambda_k>0$.
It is straightforward to verify that the operator $B_k$ defined in \eqref{eqn:BFGS} is self-adjoint, positive-definite, and verifies the secant equation: $B_k[s_k]=y_k$.

We can select the scalar $\lambda_k>0$ by following the Barzilai--Borwein method. Specifically, let $0<\lambda_{\min}\le\lambda_{\max}$; we can define $\lambda_k$ as:
\begin{equation}
\label{eqn:lambdaBB}
\lambda_k=\min\big\{\lambda_{\max},\max\{\lambda_{\min},\tilde \lambda_k\}\big\},
\end{equation}
where
\begin{equation}
\label{eqn:lambdaBB1-2}
\tilde \lambda_k=\lambda_k^{BB1}=\frac{\|s_k\|^2}{\langle s_k, y_k\rangle}, \quad \text{ or } \quad \tilde\lambda_k = \lambda_k^{BB2}=\frac{\langle s_k, y_k\rangle}{\|y_k\|^2}. 
\end{equation}

\subsection{The restart strategy}
\label{Sec:newdirection}
First, we observe that the condition $\langle s_k, y_k\rangle>0$  defined in \eqref{eqn:sy} ensures that the operator $B_k$ is positive-definite. However, it is not possible to guarantee the boundeness of the eigenvalues of the sequence of positive-definite operators $\{B_k\}$ without imposing further conditions. Hence, we cannot apply the Riemannian counterpart of \cite[Proposition 3]{lapucci2025} (cf. Remark \ref{oss:EsistenzaSoluzione}) and derive that the obtained sequence $\{d_k\}$ of gradient-related directions; we can only conclude that they are descent directions, i.e., $\langle d_k, g_k\rangle<0$ for all $k$.

To deal with this issue, we propose to follow the same restart strategy used in \cite{zhou2026}. Choose two constants $0<c_1\le c_2$, and verify whether $d_k$ satisfies the gradient-related conditions \eqref{eqn:gradrelated}. Otherwise, consider a new direction $d_k^{new}$ as
\begin{equation*}
d_k^{new}=-\lambda_kg_k,
\end{equation*}
i.e., by taking the negative gradient, rescaled by the Barzilai--Borwein factor $\lambda_k$, as defined in \eqref{eqn:lambdaBB}--\eqref{eqn:lambdaBB1-2}. Hence, in this case, our approach reduces to the one proposed by Iannazzo \& Porcelli \cite{iannazzo-porcelli2017}, but with a monotone line-search. It is straightforward to verify that the new direction is gradient-related, assuming that the constants are taken such that $c_1\le \lambda_{\min}$ and $\lambda_{\max}\le c_2$.

\begin{remark}
Condition $\langle s_k, y_k \rangle > 0$ in \eqref{eqn:sy} can be enforced by employing a line search based on the strong Wolfe conditions. However, this line-search could be significantly more expensive than Armijo's. In practice, we propose not to enforce this condition. In the cases when \eqref{eqn:sy} is not verified, we propose, as in \cite{iannazzo-porcelli2017}, to directly take the direction as
\begin{equation*}
d_k=-\lambda_{\max}g_k.
\end{equation*}
\end{remark}

\subsection{The convergent algorithm}
The proposed algorithm, based on the strategies previously discussed and with explicit formulae for $\alpha_k$ and $\beta_k$ obtained by solving system \eqref{eqn:system2x2}, is formalized in
 Algorithm \ref{alg:RGMM}. Convergence properties are stated in the following proposition.
\begin{proposition}
\label{prop:ConvergenzaAlg}
Suppose Assumptions \ref{ass:flow} and \ref{ass:assunzione} are satisfied.  Then,  Algorithm \ref{alg:RGMM} requires at most $\mathcal O(\epsilon^{-2})$ iterations, retractions, function and gradient evaluations to attain
\[ \|\grad f(x_k)\| \le \epsilon. \]
\end{proposition}

\begin{proof}
By construction, the sequence $\{d_k\}$ is gradient-related to $\{x_k\}$; hence, Assumption \ref{ass:gradientrelated} is satisfied. In addition, no extra function or gradient evaluations are performed to compute the direction $d_k$. Then, Proposition \ref{prop:Convergenza} can be applied, attaining the worst-case  complexity bound $\mathcal O(\epsilon^{-2})$.
\end{proof}

\begin{algorithm2e}
  \SetAlCapHSkip{0.22cm}
  \caption{Riemannian Gradient Method with Momentum}
  \label{alg:RGMM}
  \SetInd{0.1cm}{0.5cm}
  \SetVlineSkip{0.1cm}
  \SetNlSkip{0.2cm}
  \SetNoFillComment
  \KwIn{$x_0\in\mathcal{M}$, $\gamma\in(0,1)$, $\delta\in(0,1)$, $\epsilon>0,$ $0<c_1 \le \lambda_{\min} \le \lambda_0 \le \lambda_{\max} \le c_2$}
  \KwOut{approximate stationary point $x_k$}
  $k \leftarrow 0$\\
  $f_k \leftarrow f(x_k)$\\
  $g_k \leftarrow \grad f(x_k)$\\

	\While{$\|g_k\|> \epsilon$}{
    \If{$k=0$}{
    $
			d_k\leftarrow -\lambda_{0}g_k
			$\\ 
    }
    \Else{
    $s_k \leftarrow \T_{x_k\leftarrow x_{k-1}}(\eta_{k-1}d_{k-1})$\\
           $y_k \leftarrow g_k-\T_{x_k\leftarrow x_{k-1}}(g_{k-1})$\\
    \texttt{\# Check curvature condition}\\
    \If{$\langle s_k, y_k\rangle\le 0$}{
    $
			d_k\leftarrow -\lambda_{\max}g_k
			$\\
    }
    \Else{
    \texttt{\# Compute coefficients and direction}\\
	   $\lambda_k\leftarrow\min\big\{\lambda_{\max},\max\{\lambda_{\min},\frac{\|s_k\|^2}{\langle s_k, y_k\rangle}\}\big\}$\\
		\label{line:alpha} $\alpha_k \leftarrow \lambda_k \frac{\|g_k\|^2\langle s_k, y_k\rangle-\langle g_k, y_k\rangle\langle g_k, s_k \rangle} {\langle s_k, y_k  \rangle\bigg(\|g_k\|^2-\frac{ \langle g_k, s_k\rangle^2}{\|s_k\|^2}\bigg)}$\\
        \label{line:beta} $\beta_k \leftarrow \frac{\alpha_k\langle g_k, y_k  \rangle - \langle g_k, s_k  \rangle}{\langle s_k, y_k  \rangle}$\\
        $
			d_k\leftarrow -\alpha_kg_k+\beta_k s_k
			$\\
        \texttt{\# Check gradient-related direction}\\    
        \If{$\langle g_k, d_k\rangle>  -c_1\|g_k\|^2$ \emph{\textbf{or}} $\|d_k\|>  c_2 \|g_k\|$ }{ $d_k \leftarrow -\lambda_k g_k$}
        }
       
        }
         \texttt{\# Armijo line-search}\\   
         \label{line:eta0}$\eta \leftarrow 1$\\
         \While{$f\big(R_{x_k}(\eta d_k)\big) > f_k + \gamma\,\eta\,\langle g_k, d_k\rangle$}{
         $\eta \leftarrow \delta\,\eta$\;
        }
         $\eta_k \leftarrow \eta$\\
        $x_{k+1} \leftarrow R_{x_k}(\eta_k d_k)$\\
        $k \leftarrow k + 1$\\
        $f_k \leftarrow f(x_k)$\\
        $g_k \leftarrow \grad f(x_k)$\\
 	}
\end{algorithm2e}

\section{Experimental results}
\label{sec:esperimenti}

Following the pseudocode in Algorithm \ref{alg:RGMM}, we implemented the proposed Riemannian Gradient Method with Momentum (\texttt{RGMM}) in MATLAB R2025b. The implementation and the scripts used for the experiments are publicly available at \url{https://github.com/DiegoScuppa/RGMM}.

As described in Section \ref{sec:approxHk}, the operator $B_k$ is chosen as the memoryless BFGS update \eqref{eqn:BFGS}, which leads to the explicit expressions for $\alpha_k$ and $\beta_k$ reported in lines \ref{line:alpha}--\ref{line:beta}. The step length $\lambda_k$ is selected according to \eqref{eqn:lambdaBB}, using one of the following \texttt{strategy} options: the \texttt{direct} formula ($\lambda_k^{BB1}$ in \eqref{eqn:lambdaBB1-2}), the \texttt{inverse} formula ($\lambda_k^{BB2}$), or the \texttt{alternate} scheme, as proposed in \cite{iannazzo-porcelli2017}. Moreover, as discussed in Section \ref{Sec:newdirection}, when the gradient-related conditions are violated, the method falls back to a Barzilai--Borwein step, as in \cite{iannazzo-porcelli2017}.

We compare \texttt{RGMM} against several solvers from the Manopt package \cite{manopt}: Riemannian Barzilai--Borwein (\texttt{RBB}), Riemannian Conjugate Gradient (\texttt{RCG}), Riemannian Trust-Region (\texttt{RTR}), and Riemannian L-BFGS (\texttt{RLBFGS}). For a fair comparison, \texttt{RGMM} was implemented by modifying Manopt’s \texttt{barzilaiborwein.m} routine from \cite{iannazzo-porcelli2017}, ensuring full compatibility with Manopt’s interface.

All experiments were run on an Intel Core Ultra~7~155H machine with 16~GB of RAM. We tested 15 problems from the Manopt distribution, using five parameter combinations per problem, for a total of 75 problem instances. Table \ref{tab:problemi} summarizes the problem set and the associated manifold structures. The first eight problems coincide with those considered in \cite{iannazzo-porcelli2017}, where only the first three parameter combinations were used.

\begin{table}
\tbl{Test problems and manifold structure. The first eight problems were also considered in \cite{iannazzo-porcelli2017}.}
{\begin{tabular}{llllc} \toprule

\textbf{Name} & \multicolumn{2}{l}{\textbf{Size}} & \textbf{Description} & \textbf{Manifold} \\ 
\midrule

\hline
\multirow{3}{*}{KM} & (a) $n = 50$ & (b) $n = 100$ & \multirow{3}{6cm}{Computing the Karcher Mean of a set of $n \times n$ positive-definite matrices.} & \multirow{3}{2cm}{SPD}  \\
& (c) $n = 200$ & (d) $n = 500$ & \\
&(e) $n = 1000$ & & \\

\hline

\multirow{5}{*}{SPCA} & 
\multicolumn{2}{l}{(a)  $n = 100,\, p = 10,\, m = 2$} & \multirow{5}{6cm}{Computing the $m$ Sparse Principal Components of a $p \times n$ matrix encoding $p$ samples of $n$ variables.} 
& \multirow{5}{2cm}{Stiefel} \\
&\multicolumn{2}{l}{(b) $n = 500,\, p = 15,\, m = 5$} \\ 
&\multicolumn{2}{l}{(c) $n = 1000,\, p = 30,\, m = 10$}  \\
&\multicolumn{2}{l}{(d) $n = 2000,\, p = 60,\, m = 20$}  \\
&\multicolumn{2}{l}{(e) $n = 5000,\, p = 1500,\, m = 50$}  \\

\hline

\multirow{3}{*}{SP} & (a)
$n = 8$ & (b) $n = 12$ & \multirow{3}{6cm}{Finding the largest diameter of $n$ equal circles that can be placed on the sphere without overlap (Sphere Packing problem).} 
& \multirow{3}{2cm}{Product of Spheres} \\  
&(c) $n = 24$ & (d) $n = 50$  \\
&(e) $n = 100$ & \\

\hline

\multirow{3}{*}{DIS} & 
(a) $n = 128$ &(b) $n = 500$ & \multirow{3}{6cm}{Finding an orthonormal basis of the Dominant Invariant 3-Subspace of an $n \times n$ matrix.} 
& \multirow{3}{2cm}{Grassmann} \\ 
&(c) $n = 1000$ &(d) $n = 2000$  \\
&(e) $n = 5000$  \\

\hline

\multirow{5}{*}{LRMC} & 
\multicolumn{2}{l}{(a) $n = 100,\, p = 10,\, m = 20$} & \multirow{5}{6cm}{Low-Rank Matrix Completion: given partial observations of an $m \times n$ matrix of rank $p$, attempting to complete it.} 
& \multirow{5}{2cm}{Fixed-rank matrices} \\ 
&\multicolumn{2}{l}{(b) $n = 500,\, p = 15,\, m = 50$} \\ 
&\multicolumn{2}{l}{(c) $n = 1000,\, p = 30,\, m = 100$}  \\
&\multicolumn{2}{l}{(d) $n = 2000,\, p = 60,\, m = 200$}  \\
&\multicolumn{2}{l}{(e) $n = 5000,\, p = 150,\, m = 500$}  \\

\hline

\multirow{3}{*}{MC} & 
(a) $n = 20$ &(b) $n = 100$ & \multirow{3}{6cm}{Max-Cut: given an $n \times n$ Laplacian matrix of a graph, finding the max-cut or an upper bound on the maximum cut value.} 
& \multirow{3}{2cm}{SPD matrices of rank 2} \\ 
&(c) $n = 300$ &(d) $n = 1000$ \\
&(e) $n = 5000$ \\

\hline

\multirow{5}{*}{TSVD} & 
\multicolumn{2}{l}{(a) $n = 60,\, m = 42,\, p = 5$} & \multirow{5}{6cm}{Computing the SVD decomposition of an $m \times n$ matrix truncated to rank $p$.} 
& \multirow{5}{2cm}{Grassmann}  \\ 
&\multicolumn{2}{l}{(b) $n = 100,\, m = 60,\, p = 7$} \\ 
&\multicolumn{2}{l}{(c) $n = 200,\, m = 70,\, p = 8$} \\
&\multicolumn{2}{l}{(d) $n = 500,\, m = 80,\, p = 10$} \\
&\multicolumn{2}{l}{(e) $n = 1000,\, m = 100,\, p = 15$} \\

\hline

\multirow{5}{*}{GP} & 
\multicolumn{2}{l}{(a) $m = 10,\, N = 50$} & \multirow{5}{6cm}{Generalized Procrustes: rotationally align clouds of points. Data: matrix $A \in \mathbb{R}^{3 \times m \times N}$, each slice is a cloud of $m$ points in $\mathbb{R}^3$.} 
& \multirow{5}{2cm}{Product of rotations with Euclidean space for $A$} \\ 
&\multicolumn{2}{l}{(b) $m = 50,\, N = 100$} \\ 
&\multicolumn{2}{l}{(c) $m = 50,\, N = 500$} \\
&\multicolumn{2}{l}{(d) $m = 100,\, N = 1000$} \\
&\multicolumn{2}{l}{(e) $m = 500,\, N = 5000$} \\

\hline

\multirow{5}{*}{RIC} & 
\multicolumn{2}{l}{(a) $N = 10,\, K = 2$} & \multirow{5}{6cm}{Radio Interometric Calibration: compute the gain matrices of $N$ stations with $K$ receivers.} 
& \multirow{5}{2cm}{Symmetric fixed-rank complex matrices} \\ 
&\multicolumn{2}{l}{(b) $N = 20,\, K = 3$} \\ 
&\multicolumn{2}{l}{(c) $N = 50,\, K = 5$} \\
&\multicolumn{2}{l}{(d) $N = 100,\, K = 10$} \\
&\multicolumn{2}{l}{(e) $N = 200,\, K = 20$} \\

\hline

\multirow{3}{*}{ESVD} & 
(a) $N = 10$ &(b) $N = 100$ & \multirow{3}{6cm}{Essential SVD: solves the problem $\sum_{i=1}^N ||E_i-A_i||^2$ ($E_i$ essential matrices).} 
& \multirow{3}{2cm}{Essential matrices} \\ 
&(c) $N = 1000$ &(d) $N = 10000$ \\
&(e) $N = 100000$ \\

\hline

\multirow{3}{*}{DSD} & 
(a) $n = 50$ &(b) $n = 100$  & \multirow{3}{6cm}{Doubly Stochastic Denoising: Find a doubly stochastic matrix closest to a given matrix, in Frobenius norm.} 
& \multirow{3}{2cm}{Multinomial symmetric matrices} \\ 
&(c) $n = 200$ &(d) $n = 500$ \\
&(e) $n = 1000$ \\

\hline

\multirow{3}{*}{ESDP} & 
(a) $n = 100$ & (b) $n = 200$ & \multirow{3}{6cm}{Elliptope SDP: semidefinite programming (SDP) with unit diagonal constraints.} 
& \multirow{3}{2cm}{Product of Spheres} \\ 
&(c) $n = 500$ &(d) $n = 1000$ \\
&(e) $n = 2000$ \\ 

\hline

\multirow{5}{*}{PDIM} & 
\multicolumn{2}{l}{(a) $n = 2,\, m = 20$} & \multirow{5}{6cm}{Positive-definite Intrinsic Mean: computing an intrinsic mean of a set of $m$ positive-definite matrices with size $n \times n$.} 
& \multirow{5}{2cm}{SDP matrices} \\ 
&\multicolumn{2}{l}{(b) $n = 5,\, m = 50$} \\
&\multicolumn{2}{l}{(c) $n = 10,\, m = 100$} \\
&\multicolumn{2}{l}{(d) $n = 20,\, m = 200$} \\
&\multicolumn{2}{l}{(e) $n = 50,\, m = 500$} \\

\hline

\multirow{3}{*}{GEC} & 
(a) $n = 32$ & (b) $n = 64$ & \multirow{3}{6cm}{Finding an orthonormal basis of the Dominant Invariant 3-Subspace of $B^{-1}A$, where $A$ and $B$ are $n \times n$ matrices.} 
& \multirow{3}{2cm}{Grassmann generalized} \\ 
&(c) $n = 128$ & (d) $n = 256$  \\
&(e) $n = 512$  \\

\hline

\multirow{5}{*}{NE} & 
\multicolumn{2}{l}{(a) $n = 50,\, k = 5$} & \multirow{5}{6cm}{Nonlinear Eigenspace: solving the nonlinear eigenvalue problem.} 
& \multirow{5}{2cm}{Grassman} \\
& \multicolumn{2}{l}{(b) $n = 100,\, k = 10$} \\
& \multicolumn{2}{l}{(c) $n = 200,\, k = 20$} \\
& \multicolumn{2}{l}{(d) $n = 500,\, k = 50$} \\
& \multicolumn{2}{l}{(e) $n = 1000,\, k = 100$} \\
\bottomrule
\end{tabular}}
\label{tab:problemi}
\end{table}

Default solver parameters were used unless stated otherwise. For \texttt{RCG}, we employed the modified Hestenes--Stiefel update for $\beta_k$, which is the default choice and proved to be the most effective in our preliminary tests. For \texttt{RTR}, even when an analytical Hessian was available, we used finite-difference approximations in order to focus the comparison on function and gradient evaluations, as well as CPU time and iteration counts.

For \texttt{RGMM}, we set $c_1 = 10^{-9}$ and $c_2 = 10^{9}$. Except for $\lambda_0$, which was fixed to~$1$, all remaining parameters were set to the default values of \texttt{RBB}: $\lambda_{\min} = 10^{-3}$, $\lambda_{\max} = 10^{3}$, $\gamma = 10^{-4}$, $\delta = 0.5$, and \texttt{strategy}=\texttt{direct}. Unlike \texttt{RBB}, our method employs a monotone line-search, as preliminary experiments indicated that a non-monotone strategy did not yield better performance.

For completeness, we note that, instead of always initializing the line-search with $\eta = 1$ in line \ref{line:eta0}, we allow smaller initial values when the step size $\|d_k\|$ is significantly larger than in previous iterations. This safeguard was introduced because preliminary tests showed that overly large initial steps can substantially increase the number of function evaluations and retractions required by the line-search.

Each of the 75 problems was solved starting from 10 random initial points $x_0$, for a total of 750 instances. All solvers were initialized from the same $x_0$ and terminated when the norm of the Riemannian gradient was reduced by a factor of $10^{-6}$, that is, when \texttt{tolgradnorm}$ = 10^{-6}\|\grad f(x_0)\|$. Additional stopping criteria were \texttt{maxiter}$ = 5\cdot 10^{4}$, \texttt{maxtime}$ = 600$ seconds, and \texttt{minstepsize}$ = 10^{-10}$. A solver was considered to have failed on a given instance if it terminated due to any of these additional criteria. For each solver, averages over the 10 runs were computed for CPU time, iterations, function evaluations, gradient evaluations, and number of failures; detailed results are reported in the GitHub repository.

We also report performance profiles \cite{performance}. For a solver $S$ and a problem instance $P$, let $t_P^S$ denote the computational cost of solver $S$ on $P$ with respect to a chosen metric (e.g., CPU time), with $t_P^S = \infty$ if $S$ fails on~$P$. Let $t_P = \min_S t_P^S$. For $\tau \ge 1$, the performance profile of solver $S$ is defined as
\[
\pi_S(\tau) = \frac{\#\{\text{instances } P : t_P^S / t_P \le \tau\}}{\#\{\text{instances } P\}}.
\]
Here, $\pi_S(1)$ represents the fraction of instances for which $S$ is the best solver, while $\pi_S(\tau)$ for larger values of $\tau$ measures how often $S$ performs within a factor $\tau$ of the best solver.

\begin{table}
\tbl{Percentages of problems solved in the best CPU-time ($\pi_S(1)$), and successfully solved  ($\pi_S(\infty)$).}
{
\begin{tabular}{lrr} 
\toprule
\textbf{Solver} & $\pi_S(1)$ & $\pi_S(\infty)$ \\
\midrule
\texttt{RGMM} & $33.4\%$ & $98.1\%$ \\
\texttt{RBB} & $28.9\%$ & $98.5\%$ \\
\texttt{RCG} & $13.1\%$ & $95.3\%$ \\
\texttt{RTR} & $16.7\%$ & $100.0\%$ \\
\texttt{RLBFG} & $7.9\%$ & $98.8\%$ \\
\bottomrule
\end{tabular}
}
\label{tab:perfprof}
\end{table}

Figure \ref{fig:perf_time} shows the CPU-time performance profile for $\tau \in [1,10^2]$. As reported in Table \ref{tab:perfprof}, \texttt{RGMM} is the best (i.e., has the largest $\pi_S(1)$) on the majority of instances. Moreover, for $\tau$ between $1$ and $8$, \texttt{RGMM} attains the highest profile, indicating superior robustness across a wide range of tolerances. For $\tau \ge 10^2$ the performance profile shows constant values, which stand for the numbers of instances solved by the respective solvers, reported in Table \ref{tab:perfprof}.
Specifically,
\begin{itemize}
\item \texttt{RGMM} failed on all instances for one problem out of 75 (GEC (e)); in addition, it failed on two instances of GEC (c), and two instances of GEC (d);
\item \texttt{RBB} failed on all instances for one problem (RIC (e)); in addition, it failed on one instance of GEC (e);
\item \texttt{RCG} failed on all instances for two problems (GEC (e) and KM (e)); in addition, it failed on eight instances of GEC (c), five instances of GEC (d), one instance of SP (a), and one instance of SP (b);
\item \texttt{RTR} solved all instances of all problems;
\item \texttt{RLBFGS} solved at least one instance of all problems, failing on one instance of GEC (c), eight instances of GEC (e), and one instance of KM (e).
\end{itemize}

Figure \ref{fig:perf_iter_f} reports profiles for iterations and function evaluations for \texttt{RGMM}, \texttt{RBB}, and \texttt{RCG}. \texttt{RGMM} attains the lowest number of iterations in $52.0\%$ of instances, and the lowest number of function evaluations in $49.3\%$ of instances. Moreover, it attains the best performance in terms of both metrics for all values of $\tau$, except those sufficiently large, reflecting the fact that \texttt{RBB} solves slightly more instances than \texttt{RGMM}.

Overall, the experimental results indicate that \texttt{RGMM} is a robust and competitive method. It typically requires the fewest iterations and function evaluations and exhibits strong CPU-time performance across most tested instances. In particular, \texttt{RGMM} is the fastest solver on approximately $33.4\%$ of the instances and achieves the highest performance profile for $\tau \in [1,8]$, highlighting its consistency across a wide range of tolerances. Its failure rate is negligible and comparable to that of the competing solvers.

\begin{figure}
  \centering
  \includegraphics[width=0.8\textwidth]{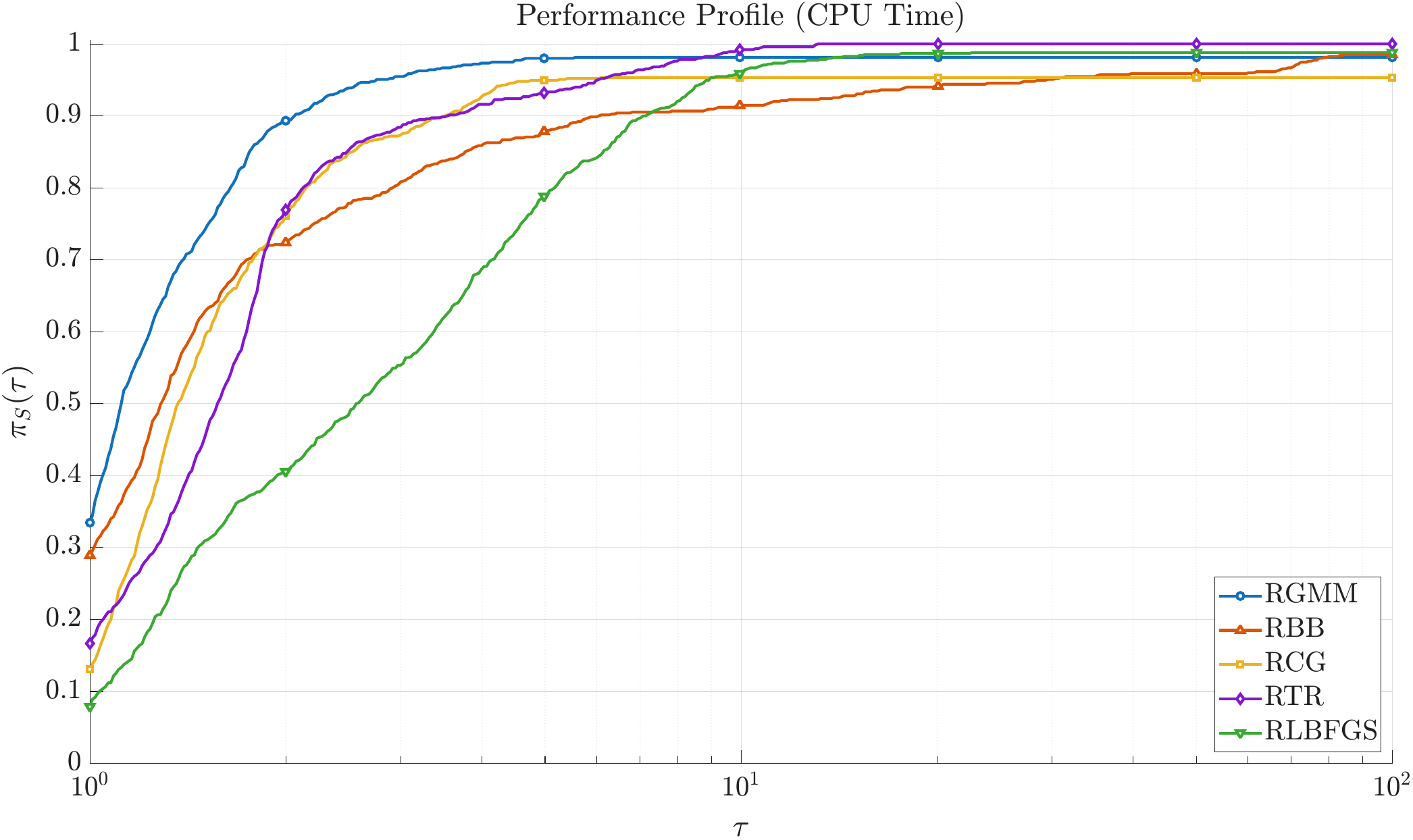}
  \caption{CPU time performance profile.}
  \label{fig:perf_time}
\end{figure}

\begin{figure}
\centering
\subfloat[Iterations performance profile.]{%
\resizebox*{7cm}{!}{\includegraphics{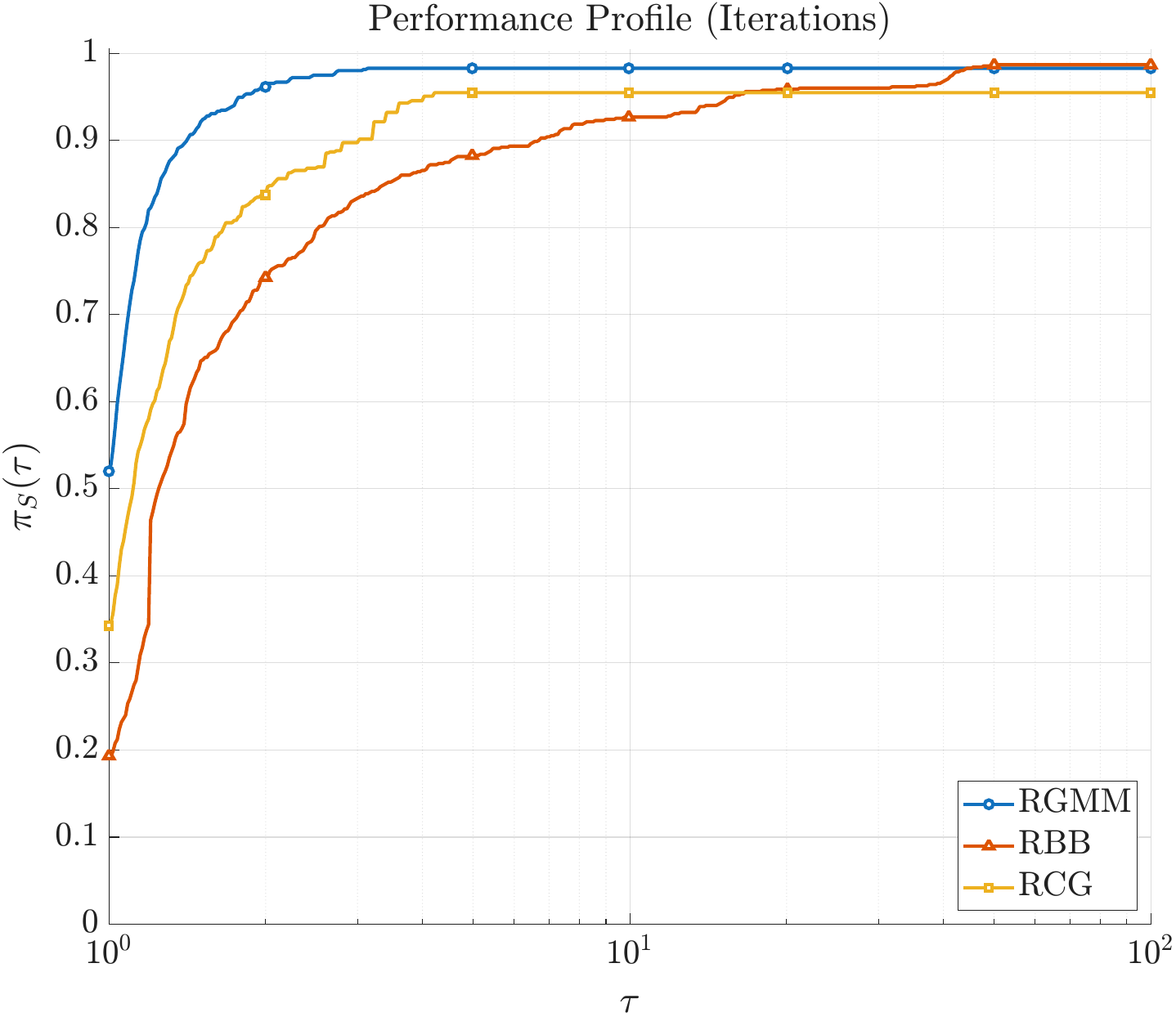}}}\hspace{5pt}
\subfloat[Function evaluations performance profile.]{%
\resizebox*{7cm}{!}{\includegraphics{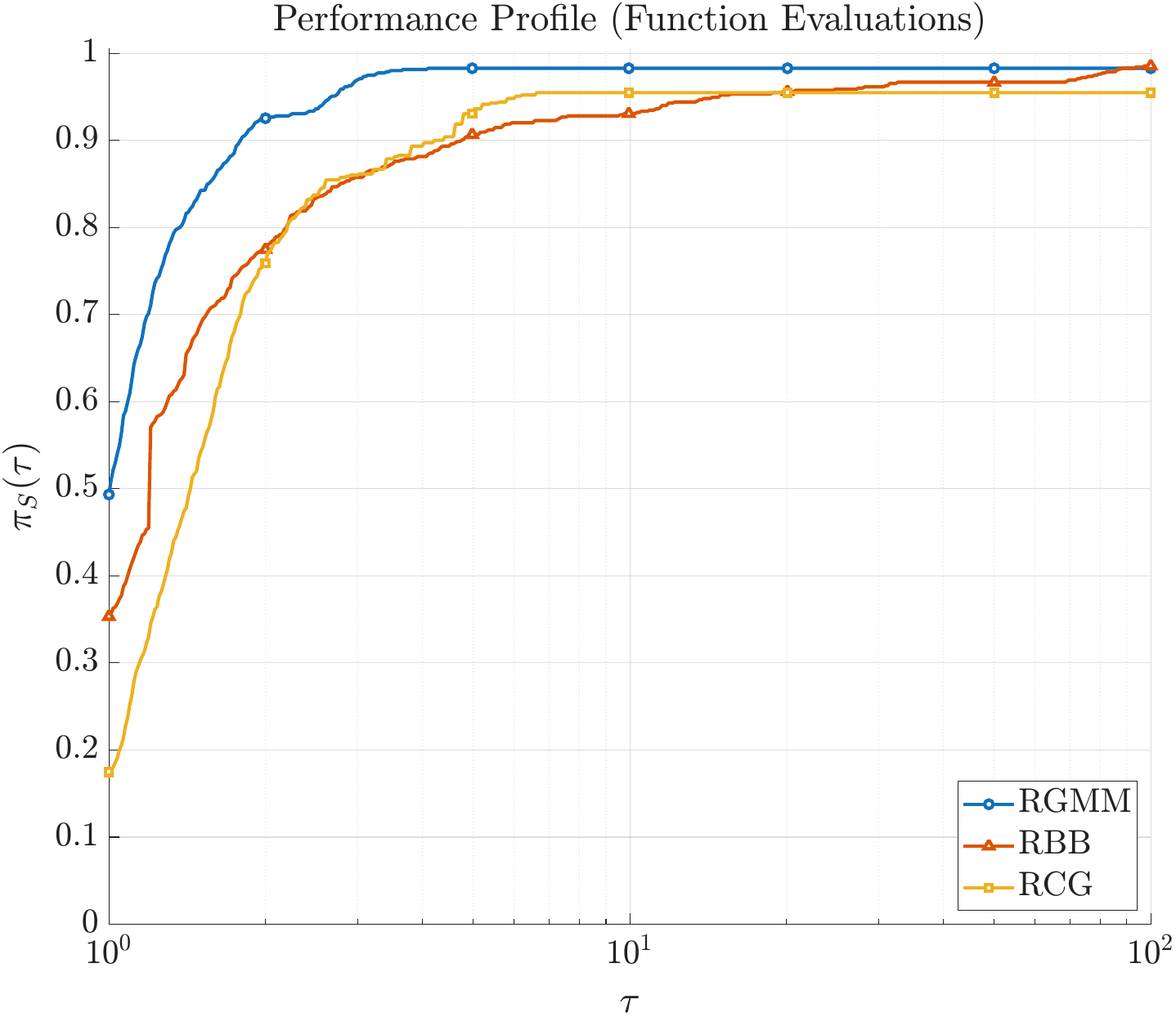}}}
\caption{Iterations and function evaluations performance profile (\texttt{RGMM}, \texttt{RBB}, \texttt{RCG}).}
    \label{fig:perf_iter_f}
\end{figure}

Finally, with $c_1=10^{-9}$ and $c_2=10^9$ the gradient-related conditions were satisfied in all tested runs, except for one iteration in three instances of the problem SP (a). Intuitively, choosing $c_1$ sufficiently small and $c_2$ sufficiently large makes these conditions nearly equivalent to a standard descent requirement, which is always satisfied in our method. Moreover, the curvature condition  $\langle s_k,y_k \rangle>0$ was rarely violated, only during a total of $514$ iterations out of $107 951$  (less than $0.5\%$); therefore, the safeguards described in Section \ref{Sec:newdirection} act mainly as theoretical protections.

\section{Conclusions}
\label{sec:conclusioni}

In this work, we presented a new Riemannian gradient method with momentum, adapting the recent approach developed by Lapucci et al.\ for the unconstrained case \cite{lapucci2025}. We studied global convergence to stationary points, with a worst-case complexity bound of $\mathcal O(\epsilon^{-2})$, under reasonable hypotheses that are milder than those presented in \cite{RegolarizzazioneCubica}. We propose a practical and effective way to choose the operator $B_k$ that does not require explicit knowledge of the Hessian or extra evaluations of gradients or functions and retractions, making this choice effective; however, in the future, other proposals may be investigated. We tested our method on a collection of problems available in the Manopt package, comparing it with existing solvers. Our algorithm was the fastest solver in the majority of problems and the most robust, demonstrating its reliability and making it a reasonable choice for solving optimization problems on Riemannian manifolds.

\section*{Acknowledgments}
The authors wish to express their gratitude to Professor Marco Sciandrone for his valuable suggestions, his careful reading of the manuscript, and his encouragement throughout the development of this work.

\section*{Disclosure statement}
The authors report there are no competing interests to declare.

\section*{Funding}
No funding was received for conducting this study.

\section*{Code Availability Statement}
The code developed for the experimental part of this paper is publicly available at \url{https://github.com/DiegoScuppa/RGMM}.

\bibliographystyle{tfs}
\bibliography{bibliography}

\appendix

\section{Proofs}
\label{app:proofs}

In this appendix, we provide the proof of Proposition \ref{prop:Convergenza}.
First, we prove a key property of the Armijo line-search procedure.

\medskip

\begin{lemma}
\label{lemma:Armijo}
Suppose Assumptions \ref{ass:gradientrelated}-\ref{ass:assunzione} are satisfied (with constants $0< c_1\le c_2$, and $L>0$). Let $\delta\in(0,1)$, $\gamma\in(0,1)$. Then, the Armijo line-search terminates in a finite number of iterations, finding a parameter $\eta_k=\delta^{j_k}\in (0,1]$ for some $j_k\in \mathbb N$. Moreover,
\begin{enumerate}
\item if $L\le2(1-\gamma)\frac{c_1}{c_2^2}$, then $j_k=0$ (i.e., $\eta_k=1$);
\item otherwise,
\begin{equation*}
j_k\le \bigg\lfloor\log_{\delta} \bigg(\frac{2}{L}(1-\gamma)\frac{c_1}{c_2^2}\bigg)\bigg\rfloor +1.
\end{equation*}
\end{enumerate}

\end{lemma}

\medskip

\begin{proof}
From Assumption \ref{ass:assunzione}, for all $k$ and all $j\in \mathbb N$,
\begin{equation}
\label{eqn:lemmaArmijo1}
f(\R_{x_k}(\delta^jd_k)) \le f(x_k) + \delta^j\langle \grad f(x_k), d_k \rangle + \frac{L}{2}\delta^{2j}\|d_k\|^2.
\end{equation}
For all $k$, define the (possibly empty) set of the unsuccessful trials in the Armijo line-search:
\begin{equation*}
J_k:=\big\{j\in\mathbb N \text{ such that } f(\R_{x_k}(\delta^jd_k)) > f(x_k) + \gamma\delta^j\langle \grad f(x_k), d_k \rangle\big\}.
\end{equation*}
Then, using also \eqref{eqn:lemmaArmijo1}, for all $j \in J_k$, we have
\begin{equation}
\label{eqn:lemmaArmijo2}
(\gamma-1) \delta^j \langle \grad f(x_k), d_k \rangle < \frac{L}{2}\delta^{2j}\|d_k\|^2.
\end{equation}
Hence, from \eqref{eqn:gradrelated} and \eqref{eqn:lemmaArmijo2}, we have:
\begin{equation*}
(1-\gamma) c_1 \|g_k\|^2 < \frac{L}{2}\delta^{j} c_2^2\|g_k\|^2 .
\end{equation*}
In particular,
\begin{equation*}
\delta^j>\frac{2}{L}(1-\gamma)\frac{c_1}{c_2^2}.
\end{equation*}
If $L\le2(1-\gamma)\frac{c_1}{c_2^2}$, we have $\delta^j > 1$, i.e., $j<0$. Thus, the set $J_k$ is empty. As a consequence, the Armijo condition is satisfied with $\eta_k=\delta^0=1$.
Otherwise, every $j\in J_k$ must satisfy $0\le j < \log_{\delta} \frac{2}{L} (1-\gamma)\frac{c_1}{c_2^2}$. Therefore, the Armijo condition is satisfied with $\eta_k=\delta^{j_k}$, where $j_k$ is at most $\lfloor \log_{\delta} (\frac{2}{L} (1-\gamma)\frac{c_1}{c_2^2})\rfloor +1$.
\end{proof}

\medskip

\begin{proof}[Proof of Proposition \ref{prop:Convergenza}]
Let $\epsilon>0$. Exploiting \eqref{eqn:gradrelated1}, Assumption \ref{ass:flow}, and the Armijo sufficient reduction, for all $K\in \mathbb N$, we have:
\begin{equation}
\label{eqn:Convergenza1}
\begin{split}
\sum_{k=0}^{K-1}\|\grad f(x_k)\|^2 \le \sum_{k=0}^{K-1}\big(-\frac{1}{c_1}\langle \grad f(x_k),d_k\rangle\big) \le 
\\ \le\sum_{k=0}^{K-1} 
\frac{1}{\gamma c_1 \eta_k} \big(f(x_k)-f(x_{k+1})\big) \le \frac{1}{\gamma c_1 \tilde{\eta}}\big(f(x_0)-f_{low}\big),
\end{split}
\end{equation}
where $\tilde \eta$ is a quantity such that $0<\tilde \eta \le \eta_k$ for all $k$ (it exists from Lemma \ref{lemma:Armijo}). Then, taking $K=k_{\epsilon}$, we have that $\|\grad f(x_k)\|>\epsilon$ for all $k=0,\dots,k_{\epsilon}-1$; hence:
\begin{equation}
\label{eqn:Convergenza2}
k_{\epsilon}\cdot\epsilon^2  <\sum_{k=0}^{k_{\epsilon}-1}\|\grad f(x_k)\|^2.
\end{equation}
From \eqref{eqn:Convergenza1} and \eqref{eqn:Convergenza2}, it follows that:
\begin{equation*}
        k_\epsilon\le \frac {\big(	f(x_0) - f_{low}\big)}{ \gamma c_1\tilde{\eta}}\, \epsilon^{-2}.
    \end{equation*}
From Lemma \ref{lemma:Armijo}, if $L\le2(1-\gamma)\frac{c_1}{c_2^2}$, then $\eta_k=1$ for all $k$; hence, $\tilde \eta$ can be chosen as $1$. Moreover, for all iterations, only one function evaluation and one retraction are performed. Hence, \eqref{eqn:caso1} follows.
Otherwise, again from Lemma \ref{lemma:Armijo}, 
\begin{equation}
\label{eqn:propArmijo1}
\eta_k\ge2\frac{1-\gamma}{L}\frac{c_1}{c_2^2}\delta,
\end{equation}
then, we can take $\tilde \eta$ as the RHS of \eqref{eqn:propArmijo1}. Then, \eqref{eqn:caso2a} follows. Moreover, at every iteration, at most 
\begin{equation*}
\bigg\lfloor\log_{\delta} \bigg(\frac{2}{L}(1-\gamma)\frac{c_1}{c_2^2}\bigg)\bigg\rfloor +2.
\end{equation*}
function evaluations and retractions are performed. Hence, \eqref{eqn:caso2b} follows.
Moreover, the inequalities in \eqref{eqn:Convergenza1} are also satisfied taking the limit for $K \to \infty$. Hence, recalling the definition of $ni_{\epsilon}$,
\begin{equation*}
ni_{\epsilon}\cdot\epsilon^2  <\sum_{k : \|\grad f(x_k)\|>\epsilon}\|\grad f(x_k)\|^2 \le \sum_{k=0}^{\infty}\|\grad f(x_k)\|^2\le \frac{1}{\gamma c_1 \tilde{\eta}}\big(f(x_0)-f_{low}\big).
\end{equation*}
Hence,
\begin{equation*}
        ni_\epsilon\le \frac {\big(	f(x_0) - f_{low}\big)}{ \gamma c_1\tilde{\eta}}\, \epsilon^{-2}.
    \end{equation*}
Thus, \eqref{eqn:caso1} and \eqref{eqn:caso2a} are also satisfied for $ni_{\epsilon}$.
\end{proof}

\end{document}